\documentclass{amsart}

\usepackage{amsmath}
\usepackage{amsfonts}
\usepackage{bbding}
\usepackage{amssymb,enumerate}
\usepackage{amsthm,stmaryrd}
\usepackage[all]{xy}
\usepackage[draft=false]{hyperref}
\usepackage{tikz}
\usetikzlibrary{matrix,arrows,decorations.pathmorphing}
\usepackage{comment}
\usepackage[obeyDraft,color=green!40]{todonotes}
\theoremstyle{plain}
\newtheorem{lem}{Lemma}[section]
\newtheorem{cor}[lem]{Corollary}
\newtheorem{prop}[lem]{Proposition}
\newtheorem{thm}[lem]{Theorem}

\theoremstyle{definition}
\newtheorem{defn}[lem]{Definition}
\newtheorem{ex}[lem]{Example}

\newtheorem{fact}[lem]{Fact}

\newtheorem{rem}[lem]{Remark}

\renewenvironment{proof}{\vspace{1ex}\noindent{\textbf{Proof:}}\hspace{0.5em}}
{\hfill\qed\vspace{1ex}}

\newcommand{\A}{\mathcal{A}}
\newcommand{\D}{\mathcal{D}}

\newcommand{\cat}[1]{\mathcal{#1}}

\newcommand{\catp}{\cat{P}}
\newcommand{\catf}{\cat{F}}
\newcommand{\cati}{\cat{I}}
\newcommand{\cata}{\cat{A}}

\newcommand{\catb}{\cat{B}}

\newcommand{\catac}{\cat{A}_C}

\newcommand{\catbc}{\cat{B}_C}

\newcommand{\catic}{\cat{I}_C}

\newcommand{\catpc}{\cat{P}_C}
\newcommand{\catfc}{\cat{F}_C}

\newcommand{\pd}{\operatorname{pd}}

\newcommand{\id}{\operatorname{id}}	
\newcommand{\fd}{\operatorname{fd}}

\newcommand{\pcpd}{\catpc\text{-}\pd}
\newcommand{\fcpd}{\catfc\text{-}\pd}

\newcommand{\icid}{\catic\text{-}\id}

\newcommand{\ciicid}{\mathrm{CI}\text{-}\icid}
\newcommand{\cifcpd}{\mathrm{CI}\text{-}\fcpd}

\newcommand{\ext}{\operatorname{Ext}}	
\newcommand{\rhom}{\mathbf{R}\!\operatorname{Hom}}	
\newcommand{\lotimes}{\otimes^{\mathbf{L}}}

\newcommand{\Hom}{\operatorname{Hom}}	
\newcommand{\coker}{\operatorname{Coker}}
\newcommand{\spec}{\operatorname{Spec}}

\newcommand{\tor}{\operatorname{Tor}}

\newcommand{\Ker}{\operatorname{Ker}}

\newcommand{\ideal}[1]{\mathfrak{#1}}
\newcommand{\m}{\ideal{m}}

\newcommand{\p}{\ideal{p}}

\newcommand{\bbz}{\mathbb{Z}}

\newcommand{\xra}{\xrightarrow}

\renewcommand{\geq}{\geqslant}
\renewcommand{\leq}{\leqslant}
\renewcommand{\ker}{\Ker}
\renewcommand{\hom}{\Hom}

\begin{document}

\title[Homological Dimensions and Semidualizing Complexes]{Homological Dimensions with Respect to a Semidualizing Complex}

\author{Jonathan Totushek}

\thanks{This material is based on work supported by North Dakota EPSCoR and National Science Foundation Grant EPS-0814442}



\maketitle

\begin{abstract}
	In this paper we build off of Takahashi and White's $\catpc$-projective dimension and $\catic$-injective dimension to define these dimensions for when $C$ is a semidaulizing complex. We develop the framework for these homological dimensions by establishing base change results and local-global behavior. Furthermore, we investigate how these dimensions interact with other invariants.
\end{abstract}



\section{Introduction}\label{141111:2}

Let $R$ be a commutative noetherian ring. The projective, flat, and injective dimensions of an $R$-module $M$ are now classical invariants that are important for studying M and $R$. These dimensions were later generalized for $R$-complexes by Foxby \cite{foxby:bcfm} and many useful results about dimensions for modules also hold true for complexes.

A finitely generated $R$-module $C$ is \textit{semidualizing} if $R\cong \hom_R(C,C)$ and $\ext^{\geq 1}_R (C,C)=0$.  Takahashi and White \cite{takahashi:hasm} defined, for a semidualizing $R$-module $C$, the $\catp_C$-projective and $\cati_C$-injective dimensions. The \textit{$\catpc$-projective dimension} of an $R$-module $M$ ($\pcpd_R(M)$) is the length of the shortest resolution of $M$ by modules of the form $C\otimes_R P$ where $P$ is a projective module. They define \textit{$\catic$-injective dimension} ($\icid_R(M)$) dually, and one defines the \textit{$\catfc$-projective dimension} ($\fcpd_R(M)$) similarly. We extend these constructions to the realm of $R$-complexes. Note that we work in the derived category $\D(R)$. See Section \ref{141111:1} for some background and notation on this subject.

A complex $C\in \D_{\operatorname{b}}^{\operatorname{f}}(R)$ is \textit{semidualizing} if the natural homothety morphism $\chi^R_C:R\to \rhom_R(C,C)$ is an isomorphism in $\D(R)$. To understand the $\catp_C$-projective, $\catf_C$-projective, and $\cati_C$-injective dimensions in this context, we use the following result; see Theorem \ref{141111:6} below.

\begin{thm}
\label{141111:3}
	Let $X\in \D_{\operatorname{b}}(R)$, and let $C$ be a semidualizing $R$-complex.
	\begin{enumerate}[\rm(a)]
	\item We have $\pd_R(\rhom_R(C,X))<\infty$ if and only if there exists $Y\in \D_{\operatorname{b}}(R)$ such that $\pd_R(Y)<\infty$ and $X\simeq C\lotimes_R Y$ in $\D(R)$. When these conditions are satisfied, one has $Y\simeq \rhom_R(C,X)$ and $X\in \catbc(R)$.\label{141111:3a}
	\item We have $\fd_R(\rhom_R(C,X))<\infty$ if and only if there exists $Y\in \D_{\operatorname{b}}(R)$ such that $\fd_R(Y)<\infty$ and $X\simeq C\lotimes_R Y$ in $\D(R)$. When these conditions are satisfied, one has $Y\simeq \rhom_R(C,X)$ and $X\in \catbc(R)$.\label{141111:3b}
	\item We have $\id_R(C\lotimes_R X)<\infty$ if and only if there exists $Y\in \D_{\operatorname{b}}(R)$ such that $\id_R(Y)<\infty$ and $X\simeq \rhom_R(C,Y)$ in $\D(R)$. When these conditions are satisfied, one has $Y\simeq C\lotimes_R X$ and $X\in \catac(R)$.\label{141111:3c}
	\end{enumerate}
\end{thm}
With this in mind, we define e.g., $\pcpd_R(X) := \sup(C) + \pd_R(\rhom_R(C,X)$; thus $\pcpd_R(X)<\infty$ if and only if $X$ satisfies the equivalent conditions of Theorem \ref{141111:3}\eqref{141111:3a}. We define $\fcpd_R(X)$ and $\icid_R(X)$ similarly. 

In Section \ref{141111:4} we develop the foundations of these homological dimensions. For instance, we establish finite flat dimension base change  (\ref{141021:1}) and local-global principles (\ref{140623:10}-\ref{140623:12}). Also in Theorem \ref{141112:1} we show how these notions naturally augment Foxby Equivalence. In Section \ref{141111:5} we establish some stability results and the following; see Theorem \ref{140503:1}.

\begin{thm}
\label{141117:1}
	Assume $R$ has a dualizing complex $D$ and let $X\in \mathcal{D}_{\rm{b}}(R)$. Then $\fcpd_R(X)<\infty$ if and only if $\mathcal{I}_{C^{\dagger}}\text{-}\id_R(X)<\infty$ where $C^{\dagger} = \rhom_R(C,D)$.
\end{thm}

This result is key for the work in \cite{sather:fohdwrtasc}.

\section{Background}\label{141111:1}

Throughout this paper $R$ and $S$ are commutative noetherian rings with identity and $C$ is a semidualizing $R$-complex.

We work in the derived category $\D(R)$ of complexes of $R$-modules, indexed homologically (see e.g. \cite{gelfand:moha,hartshorne:rad}). A complex $X\in \D(R)$ is \textit{homologically bounded if} $\operatorname{H}_i(X) = 0$ for all $|i|\gg 0$ and $X$ is \textit{homologically finite} if $\oplus_i \operatorname{H}_i(X)$ is finitely generated. We denote by $\D_{\operatorname{b}}(R)$ and $\D_{\operatorname{b}}^{\operatorname{f}}(R)$ the full subcategories of $\D(R)$ consisting of all homologically bounded $R$-complexes and all homologically finite $R$-complexes, respectively. Isomorphisms in $\D(R)$ are identified by the symbol $\simeq$.

For $R$-complexes $X$ and $Y$, let $\inf(X)$ and $\sup(X)$  denote the infimum and supremum, respectively, of the set $\{i\in \bbz\mid \operatorname{H}_i(X) = 0\}$. Let $X\lotimes_R Y$ and $\rhom_R(X,Y)$ denote the left-derived tensor product and right-derived homomorphism complexes, respectively.

\begin{defn}
\label{140630:1}
	Let $X\in \D_+(R)$. The \textit{projective dimension} of $X$ is
	\[
		\pd_R(X) = \inf\left\{n\in \bbz \left| 
		\begin{array}{l}
			P\xra{\simeq} X \text{ where } P \text{ is a complex of projective}\\
			R\text{-modules such that } P_i = 0 \text{ for all } i>n
		\end{array}\right.
		\right\}.
	\]
	The \textit{flat dimension} ($\fd$) and \textit{injective dimension} ($\id$) are defined similarly. Let $\catp(R)$, $\catf(R)$, and $\cati(R)$ denote the full subcategories of $\D_{\operatorname{b}}(R)$ consisting of complexes of finite projective, flat, and injective dimensions, respectively.
\end{defn}

\begin{fact}[{\cite[Proposition 4.5]{avramov:hdouc}}]
\label{141121:1}
	Let $X,Y\in\D(R)$.
	\begin{enumerate}[\rm(a)]
	\item If $\id_R(Y)<\infty$, then $\fd_R(\rhom_R(X,Y))\leq \id_R(X) + \sup(Y)$.\label{141121:1a}
	\item If $\fd_R(Y)<\infty$, then $\id_R(X\lotimes_R Y)\leq \id_R(X) - \inf(Y)$.\label{141121:1b}
	\end{enumerate}
\end{fact}

The following result is for use in Section \ref{141111:5}.

\begin{lem}
\label{140821:1}
	Let $X\in \D_{\mathrm{b}}(R)$.
	\begin{enumerate}[\rm(a)]
	\item If $I$ is a faithfully injective $R$-module and $\id_R(\rhom_R(X,E))\leq n$, then we have $\fd_R(X)\leq n$.\label{140821:1a}
	\item If $F$ is a faithfully flat $R$-module and $\fd_R(X\lotimes_R F)\leq n$, then $\fd_R(X)\leq n$.\label{140821:1b}
	\item If $E$ is a faithfully injective $R$-module and $\id_R(X)\leq n$, then we have that $\fd_R(\rhom_R(X,E))\leq n$.\label{141007:1a}
	\item If $F$ is a faithfully flat $R$-module and $\id_R(X)\leq n$, then $\id_R(X\lotimes_R E)\leq n$.\label{141007:1b}
	\end{enumerate}
\end{lem}

\begin{proof}
\eqref{140821:1a} Assume that $\id_R(\rhom_R(X,E))\leq n$ and let $F\xra{\simeq} X$ be a flat resolution.
	A standard truncation argument shows that $\hom_R(\coker(\partial^F_{n+1}),E)$ is injective. Since $E$ is faithfully injective, we also have that $\coker(\partial^F_{n+1})$ is flat. Thus $\fd_R(X)\leq n$.

	The proofs of \eqref{140821:1b}, \eqref{141007:1a}, and \eqref{141007:1b} are similar.
\end{proof}

\begin{fact}[{\cite[Lemma 4.4]{avramov:hdouc}}]
\label{141009:1}
	Let $L,M,N \in \D(R)$. Assume that $L\in \D^{\operatorname{f}}_{+}(R)$.

	The natural \textit{tensor-evaluation} morphism
	\[
		\omega_{LMN}:\rhom_R(L,M)\lotimes_R N \to \rhom_R(L,M\lotimes_R N)
	\]
	is an isomorphism when $M\in \D_{-}(R)$ and either $L\in \catp(R)$ or $N\in \catf(R)$.

	The natural \textit{Hom-evaluation} morphism
	\[
		\theta_{LMN}:L\lotimes_R\rhom_R(M,N)\to \rhom_R(\rhom_R(L,M),N)
	\]
	is an isomorphism when $M\in \D_{\operatorname{b}}(R)$ and either $L\in \catp(R)$ or $N\in \cati(R)$.
\end{fact}

\begin{defn}[Foxby Classes]
\label{140701:1}

\

	\begin{enumerate}[\rm(1)]
	\item The \textit{Auslander Class} with respect to $C$ is the full subcategory $\catac(R)\subseteq \D_{\mathrm{b}}(R)$ such that a complex $X$ is in $\catac(R)$ if and only if $C\lotimes_R X\in \D_{\mathrm{b}}(R)$ and the natural morphism $\gamma^C_X:X\to \rhom_R(C,C\lotimes_R X)$ is an isomorphism in $\D(R)$. \label{140701:1,1}
	\item The \textit{Bass Class} with respect to $C$ is the full subcategory $\catbc(R)\subseteq \D_{\mathrm{b}}(R)$ such that a complex $Y$ is in $\catbc(R)$ if and only if $\rhom_R(C,Y)\in\D_{\mathrm{b}}(R)$ and the natural morphism $\xi^C_Y:C\lotimes_R \rhom_R(C,Y) \to Y$ is an isomorphism in $\D(R)$. \label{140701:1,2}
	\end{enumerate}
\end{defn}

For a generalized diagramatic version of the next result, see Theorem \ref{141112:1}.

\begin{fact}[Foxby Equivalence {\cite[Theorem 4.6]{christensen:scatac}}]
\label{13032511}
	Let $X,Y\in \D_{\operatorname{b}}(R)$.
	\begin{enumerate}[(a)]
	\item One has $X\in \mathcal{A}_{C}(R)$ if and only if $C\lotimes_R X \in \mathcal{B}_{C}(R)$.\label{13032511a}
	\item One has $Y\in \mathcal{B}_{C}(R)$ if and only if $\rhom_R(C,Y)\in \A_{C}(R)$.\label{13032511b}
	\end{enumerate}
\end{fact}

\begin{fact}[{\cite[Proposition 4.4]{christensen:scatac}}]
\label{14012006}
	Let $X\in \mathcal{D}_{\text{b}}(R)$.
	\begin{enumerate}[(a)]
		\item If $\fd_R(X)<\infty$ (e.g., $\pd_R(X)<\infty$), then $X\in \catac(R)$.
		\label{14012006a}
		\item If $\id_R(X)<\infty$, then $X\in \catbc(R)$.
		\label{14012006b}
	\end{enumerate}
\end{fact}

\section{$C$-Dimensions for Complexes}\label{141111:4}

In this section we define for the $\catp_C$-projective, $\catf_C$-projective, and $\cati_C$-injective dimensions and build their foundations.

\begin{defn}
\label{140226:1}
	Let $X\in \D_{\operatorname{b}}(R)$. 
	\begin{enumerate}[\rm(1)]
	\item The \textit{$\catpc$-projective dimension} of $X$ is defined as \label{140226:1a}
	\[
		\pcpd_R(X) = \sup(C)+\pd_R(\rhom_R(C,X)).
	\]
	\item The \textit{$\catfc$-projective dimension} of $X$ is defined as \label{140226:1b}
	\[
		\fcpd_R(X) = \sup(C)+\fd_R(\rhom_R(C,X)).
	\]
	\item The \textit{$\catic$-injective dimension} of $X$ is defined as\label{140226:1c} 
	\[
		\icid_R(X) = \sup(C)+\id_R(C\lotimes_R X).
	\]
	\end{enumerate}
	Let $\mathcal{P}_C(R)$, $\mathcal{F}_C(R)$, and $\mathcal{I}_C(R)$ denote the full subcategories of $\D_{\operatorname{b}}(R)$ of all complexes of finite $C$-projective, $C$-flat, and $C$-injective dimension, respectively.
\end{defn}

\begin{rem}\label{140623:5}
	Let $X\in \D_{\operatorname{b}}(R)$. Observe that $\sup(C)<\infty$. Hence $\pcpd_R(X)<\infty$ if and only if $\pd_R(\rhom_R(C,X))<\infty$. If $\pcpd_R(X)<\infty$, then  Fact \ref{14012006}\eqref{14012006a} implies that $\rhom_R(C,X)\in \catac(R)$ and Foxby Equivalence (\ref{13032511}) implies that $X\in \catbc(R)$. Similarly, $\fcpd_R(X)<\infty$ if and only if $\fd_R(\rhom_R(C,X))<\infty$. If $\fcpd_R(X)<\infty$, then $X\in\catbc(R)$. Also we have $\icid_R(X)<\infty$ if and only if $\id_R(C\lotimes_R X)<\infty$. Hence, if $\icid_R(X)<\infty$, then $X\in\catac(R)$.
\end{rem}

\begin{rem}
\label{140623:2}
	Let $X\in \D_{\operatorname{b}}(R)$. Note that when $C=R$ we have that $\pcpd_R(X) = \sup(R) + \pd_R(\rhom_R(R,X)) = \pd_R(X)$. Similarly in this case $\fcpd_R(X) = \fd_R(X)$ and $\icid_R(X) = \id_R(X)$.
\end{rem}

\begin{rem}\label{140623:1}
	Let $M$ be an $R$-module. When $C$ is a semidualizing $R$-module, Takahashi and White \cite[Theorem 2.11]{takahashi:hasm}, using the definition described in Section \ref{141111:2}, showed that $\pcpd_R(X) = \pd_R(\rhom_R(C,X))$. 
	Since $\sup(C) = 0$ in this case, Definition \ref{140226:1}\eqref{140226:1a} shows that our definition is consistent with the one from \cite{takahashi:hasm}. In a similar way, it can be shown that $\icid$ recovers Takahashi and White's definition in this case.
\end{rem}

The next result compares $\fcpd$ with $\pcpd$.

\begin{prop}
\label{140305:1}
	Let $X\in \mathcal{D}_{\operatorname{b}}(R)$. Then 
	\[
		\fcpd_R(X)\leq \pcpd_R(X)\leq \fcpd_R(X)+\dim(R).
	\]
	In particular if $\dim(R)<\infty$, then we have $\pcpd_R(X)<\infty$ if and only if $\fcpd_R(X)<\infty$.
\end{prop}

\begin{proof}
	Assume that $\pcpd_R(X)=n<\infty$. Then 
	\[
		\fd_R(\rhom_R(C,X))\leq \pd_R(\rhom_R(C,X)) = n - \sup(C)<\infty.
	\]
	It now follows that $\fcpd_R(X)\leq n$.

	Next assume that $\dim(R)<\infty$ and $\fcpd_R(X)=n <\infty$. By \cite{raynaud:cpptpm} we have
	\[
		\pd_R(\rhom_R(C,X))\leq \fd_R(\rhom_R(C,X))+\dim(R)=n-\sup(C)+\dim(R).
	\]
	Therefore $\pcpd_R(X)\leq \dim(R) + n$. 

\end{proof}

The following three results are versions of \cite[Theorem 2.11]{takahashi:hasm} involving a semidaulizing complex.
\begin{prop}
\label{14012004}
	Let $X\in \D_{\operatorname{b}}(R)$. Then we have 
	\[
		\pcpd_R(C\lotimes_R X) = \sup(C)+\pd_R(X).
	\]
	In particular, $\pcpd_R(C\lotimes_R X)<\infty$ if and only if $\pd_R(X)<\infty$.
\end{prop}

\begin{proof}
Let $n\in \bbz$. We prove that $\pcpd_R(C\lotimes_R X)\leq n$ if and only if $\sup(C) + \pd_R(X)\leq n$.

	For the forward implication assume that $\pcpd_R(C\lotimes_R X) \leq n$. Then by Definition \ref{140226:1}\eqref{140226:1a} we have 
	\[
		\sup(C)+\pd_R(\rhom_R(C,C\lotimes_R X)) = \pcpd_R(C\lotimes_R X) \leq n.
	\]
	Thus $\pd_R(\rhom_R(C,C\lotimes_R X))<\infty$. Fact \ref{14012006}\eqref{14012006a} implies $\rhom_R(C,C\lotimes_R X)\in \catac(R)$. By Foxby Equivalence (\ref{13032511}) we have $C\lotimes_R X\in \catbc(R)$ and $X\in\catac(R)$. Therefore we have $X\simeq \rhom_R(C,C\lotimes_R X)$ and $\pd_R(\rhom_R(C,C\lotimes_R X)) = \pd_R(X)$. Thus $\sup(C)+\pd_R(X)\leq n$.

	For the reverse implication assume that $\sup(C) + \pd_R(X)\leq n$. In particular, we have that $\pd_R(X)<\infty$. Therefore $X\in \catac(R)$ and $X\simeq \rhom_R(C,C\lotimes_R X)$. It follows that $\pd_R(X) = \pd_R(\rhom_R(C,C\lotimes_R X)$. By Definition \ref{140226:1}\eqref{140226:1a} we have 
	\[
		\pcpd_R(C\lotimes_R X) = \sup(C)+\pd_R(\rhom_R(C,C\lotimes_R X)) = \sup(C) + \pd_R(X)\leq n.
	\]
\end{proof}

The next two results are proven like Proposition \ref{14012004}.
\begin{prop}
\label{140623:3}
	Let $X\in \D_{\operatorname{b}}(R)$. Then we have 
	\[
		\fcpd_R(C\lotimes_R X)= \sup(C) + \fd_R(X).
	\]
	In particular, $\fcpd_R(C\lotimes_R X)<\infty$ if and only if $\fd_R(X)<\infty$.
\end{prop}

\begin{prop}
\label{140623:4}
	Let $X\in \D_{\operatorname{b}}(R)$. Then we have 
	\[
		\icid_R(\rhom_R(C,X))= \sup(C)+\id_{R}(C\lotimes_R X).
	\]
	In particular, $\icid_R(\rhom_R(C,X))<\infty$ if and only if $\id_R(X)<\infty$.
\end{prop}

Next, we have Theorem \ref{141111:3} from the introduction.

\begin{thm}
\label{141111:6}
	Let $X\in \D_{\operatorname{b}}(R)$.
	\begin{enumerate}[\rm(a)]
	\item We have $\pcpd_R(X)<\infty$ if and only if there exists $Y\in \D_{\operatorname{b}}(R)$ such that $\pd_R(Y) <\infty$ and $X\simeq C\lotimes_R Y$. When these conditions are satisfied, one has $Y\simeq \rhom_R(C,X)$ and $X\in \catbc(R)$.\label{141111:6a}
	\item We have $\fcpd_R(X)<\infty$ if and only if there exists $Y\in \D_{\operatorname{b}}(R)$ such that $\fd_R(Y) <\infty$ and $X\simeq C\lotimes_R Y$. When these conditions are satisfied, one has $Y\simeq \rhom_R(C,X)$ and $X\in \catbc(R)$.\label{141111:6b}
	\item We have $\icid_R(X)<\infty$ if and only if there exists $Y\in \D_{\operatorname{b}}(R)$ such that $\id_R(Y)<\infty$ and $X\simeq \rhom_R(C,Y)$. When these conditions are satisfied, one has $Y\simeq C\lotimes_R X$ and $X\in \catac(R)$.\label{141111:6c}
	\end{enumerate}
\end{thm}

\begin{proof}
	\eqref{141111:6a} For the forward implication assume that $\pcpd_R(X) <\infty$. Then by Definition \ref{140226:1}\eqref{140226:1a} we have $\pd_R(\rhom_R(C,X)) = \pcpd_R(X) - \sup(C)<\infty$. Fact \ref{14012006}\eqref{14012006a} implies that $\rhom_R(C,X)\in\catac(R)$ and Foxby Equivalence implies that $X\in \catbc(R)$. Thus $X\simeq C\lotimes_R \rhom_R(C,X)\simeq C\lotimes_R Y$ with $Y= \rhom_R(C,X)$.

	For the reverse implication assume that there exists a $Y\in \D_{\operatorname{b}}(R)$ such that $\pd_R(Y) <\infty$ and $X\simeq C\lotimes_R Y$. Then Fact \ref{14012006}\eqref{14012006a} implies that $Y\in \catac(R)$ and hence we have
	\[
		Y\simeq \rhom_R(C,C\lotimes_R Y) \simeq \rhom_R(C,X).\
	\]
	It now follows by Definition \ref{140226:1}\eqref{140226:1a} that $\pcpd_R(X)<\infty$.

	Parts \eqref{141111:6b} and \eqref{141111:6c} are proven similarly.
\end{proof}

The previous results give rise to a generalized Foxby Equivalence.

\begin{thm}[Foxby Equivalence]
\label{141112:1}
	There is a commutative diagram
\begin{center}
\begin{tikzpicture}
	\matrix[matrix of math nodes,row sep=2em, column sep=3em,
	text height=1.5ex, text depth=0.25ex]
	{ |[name=M]| \catic(R) & &  |[name=N]| \cati(R)\\
	 |[name=C]| \catac(R) & & |[name=D]| \catbc(R)\\
	 |[name=G]| \mathcal{F}(R) & & |[name=H]| \mathcal{F}_C(R)\\
	|[name=K]| \catp(R) & & |[name=L]| \catpc(R)\\};
	\path[->,font=\scriptsize]
	([yshift= 3pt]M.east) edge ([yshift= 3pt]N.west)
	([yshift= 3pt]C.east) edge node[above]{$C\lotimes_R -$} ([yshift= 3pt]D.west)
	([yshift= 3pt]G.east) edge ([yshift= 3pt]H.west)
	([yshift= 3pt]K.east) edge ([yshift= 3pt]L.west);
	\path[<-,font=\scriptsize]
	([yshift= -3pt]M.east) edge ([yshift= -3pt]N.west)
	([yshift= -3pt]C.east) edge node[below]{$\rhom_R(C,-)$} ([yshift= -3pt]D.west)
	([yshift= -3pt]G.east) edge ([yshift= -3pt]H.west)
	([yshift= -3pt]K.east) edge ([yshift= -3pt]L.west);
	\path[right hook->] 
	(M) edge (C)
	(N) edge (D);

	\path[left hook->]
	(K) edge (G)
	(G) edge (C)
	(L) edge (H)
	(H) edge (D);
\end{tikzpicture}
\end{center}
where the vertical arrows are full embeddings, and the unlabeled horizontal arrows are quasi-inverse equivalences of categories.
\end{thm}

The next result shows how $\pcpd$ and $\fcpd$ transfer along a ring homomorphism of finite flat dimension. Note that if $\varphi: R\to S$ is a ring homomorphism of finite flat dimension, then $C\lotimes_R S$ is a semidualizing $S$-complex by \cite[Theorem 5.6]{christensen:scatac} and \cite[Theorem II(a)]{frankild:rrhffd}.

\begin{prop}
\label{141021:1}
	Let $\varphi:R\to S$ be a ring homomorphism of finite flat dimension and $X\in\D_{\operatorname{b}}(R)$. Then one has
	\begin{enumerate}[\rm(a)]
	\item $\catp_{C\lotimes_R S}\text{-}\pd_{S}(X\lotimes_R S)-\sup(C\lotimes_R S)\leq \pcpd_R(X)-\sup(C)$\label{141021:1c},
	\item $\catf_{C\lotimes_R S}\text{-}\pd_{S}(X\lotimes_R S)-\sup(C\lotimes_R S) \leq \fcpd_R(X) - \sup(C)$\label{141021:1d},
	\item $\catp_{C\lotimes_R S}\text{-}\pd_{S}(X\lotimes_R S)\leq \pcpd_R(X)$\label{141021:1a}, and
	\item $\catf_{C\lotimes_R S}\text{-}\pd_{S}(X\lotimes_R S)\leq \fcpd_R(X)$.\label{141021:1b}
	\end{enumerate}
	Equality holds when $\varphi$ is faithfully flat.
\end{prop}

\begin{proof}
	\eqref{141021:1c} and \eqref{141021:1a} Assume first that $\pcpd_R(X) - \sup(C) = n<\infty$. Then $\pd_R(\rhom_R(C,X))= n$ and hence by base change we have 
	\[
		\pd_S(\rhom_R(C,X)\lotimes_R S)\leq \pd_R(\rhom_R(C,X)) = n.
	\]
	Observe by tensor-evaluation (\ref{141009:1}) and Hom-tensor adjointness, there are isomorphisms
	\begin{align*}
		\rhom_R(C,X)\lotimes_R S &\simeq \rhom_R(C,X\lotimes_R S)\\
		&\simeq \rhom_R(C,\rhom_S(S,X\lotimes_R S))\\
		&\simeq \rhom_S(C\lotimes_R S,X\lotimes_R S).
	\end{align*}
	Therefore $\pd_S(\rhom_S(C\lotimes_R S,X\lotimes_R S))\leq n$. Thus we have
	\[
		\catp_{C\lotimes_R S}\text{-}\pd_S(X\lotimes_R S) - \sup(C\lotimes_R S) \leq n = \pcpd_R(X) - \sup(C)
	\]
	that is, the inequality in \eqref{141021:1c} holds.
	
	Observe that since $\fd_R(S)<\infty$, we have $S\in \catac(R)$ and hence $\sup(C\lotimes_R S) \leq \sup(C)$ by \cite[Proposition 4.8(a)]{christensen:scatac}. Hence the inequality in \eqref{141021:1a} follows from part \eqref{141021:1c}.

	Now assume that $\varphi$ is faithfully flat. Therefore one has that $\sup(C\lotimes_R S) = \sup(C)$. Hence it suffices to show that $\catp_{C\lotimes_R S}\text{-}\pd_R(X\lotimes_R S) \geq \pcpd_R(X)$. Assume that $\catp_{C\lotimes_R S}\text{-}\pd_R(X\lotimes_R S) = n<\infty$. Then 
	\[
		\pd_S(\rhom_R(C,X)\lotimes_R S)=\pd_S(\rhom_S(C\lotimes_RS,X\lotimes_R S)) = n-\sup(C\lotimes_R S).
	\]
	Therefore we have $\pd_S(\rhom_R(C,X)\lotimes_R S)\leq n - \sup(C)$. Observe that if $P$ is an $R$-module such that $P\otimes_R S$ is projective over $S$, then $P$ is projective over $R$ by \cite[Theorem 9.6]{perry:ffdpm} and \cite{raynaud:cpptpm}. A standard truncation argument thus shows that 
	\[
		\pd_R(\rhom_R(C,X))\leq \pd_S(\rhom_R(C,X)\lotimes_R S) = n-\sup(C)
	\]
	as desired. 

	Parts \eqref{141021:1b} and \eqref{141021:1d} are proven similarly.
\end{proof}

\begin{cor}
\label{140618:1}
	Let $X\in \D_{\operatorname{b}}(R)$, and let $U\subset R$ be a multiplicatively closed subset. Then there are equalities
	\begin{enumerate}[\rm (a)]
	\item $\mathcal{P}_{U^{-1}C}\text{-}\pd_{U^{-1} R}(U^{-1}X)\leq \pcpd_R(X)$, \label{140618:1a}
	\item $\mathcal{F}_{U^{-1}C}\text{-}\pd_{U^{-1}R}(U^{-1}X)\leq \fcpd_R(X)$, \label{140618:1b}
	\item $\mathcal{I}_{U^{-1}C}\text{-}\id_{U^{-1}R}(U^{-1}X)\leq \icid_R(X)$, \label{140618:1c}
	\item $\mathcal{P}_{U^{-1}C}\text{-}\pd_{U^{-1} R}(U^{-1}X)-\sup(U^{-1}C) \leq \pcpd_R(X) - \sup(C)$, \label{140618:1d}
	\item $\mathcal{F}_{U^{-1}C}\text{-}\pd_{U^{-1}R}(U^{-1}X)- \sup(U^{-1}C)\leq \fcpd_R(X)-\sup(C)$, and \label{140618:1e}
	\item $\mathcal{I}_{U^{-1}C}\text{-}\id_{U^{-1}R}(U^{-1}X)-\sup(U^{-1}C)\leq \icid_R(X) - \sup(C)$. \label{140618:1f}
	\end{enumerate}
\end{cor}

\begin{proof}
	The map $\varphi:R\to U^{-1}R$ is flat. Hence \eqref{140618:1a}, \eqref{140618:1b}, \eqref{140618:1d}, and \eqref{140618:1e} follow from Proposition~\ref{141021:1}. Parts \eqref{140618:1c} and \eqref{140618:1f} are proven similarly to Proposition \ref{141021:1}.
\end{proof}

\begin{rem}\label{140623:9}
	Observe that to obtain the inequality in Corollary~\ref{140618:1} we need the inequality $\sup(U^{-1}C)\leq \sup(C)$ to hold. If we had defined $\pcpd_R(X)$ as $\inf(C) + \pd_R(\rhom_R(C,X))$,
	then Corollarly \ref{140618:1} would not hold because $\inf(U^{-1}C)\not\leq \inf(C)$. This is why we choose $\sup(C)$ instead of $\inf(C)$ in the definition of $\pcpd$.
\end{rem}

The next result is a local-global principal for Bass classes.

\begin{lem}
\label{141028:1}
	Let $X\in \D_{\operatorname{b}}(R)$. The following conditions are equivalent:
	\begin{enumerate}[\rm(i)]
	\item $X\in \catbc(R)$;
	\item $U^{-1}X\in \catb_{U^{-1}C}(U^{-1}R)$ all multiplicatively closed subsets $U\subset R$;
	\item $X_{\p}\in \catb_{C_{\p}}(R_{\p})$ for all $\p\in \spec(R)$;
	\item $X_{\p}\in \catb_{C_{\p}}(R_{\p})$ for all $\p\in \operatorname{Supp}(R)$;
	\item $X_{\m}\in \catb_{C_{\m}}(R_{\m})$ for all $\m\in \operatorname{Max}(R)$; and
	\item $X_{\m}\in \catb_{C_{\m}}(R_{\m})$ for all $\m\in \operatorname{Supp}(R)\cap\operatorname{Max}(R)$.
	\end{enumerate}
\end{lem}

\begin{proof}
	The implications (i) $\Rightarrow$ (ii) $\Rightarrow$ (iii) $\Rightarrow$ (iv) $\Rightarrow$ (vi) and (iii) $\Rightarrow$ (v) $\Rightarrow$ (vi) follow from definitions. We prove (v) $\Rightarrow$ (i) and (vi) $\Rightarrow$ (v).

	For the implication (v) $\Rightarrow$ (i), assume $X_{\m}\in\catb_{C_{\m}}(R_{\m})$ for all $\m\in \operatorname{Max}(R)$. We use the following commutative diagram in $\D(R)$:
	\begin{center}
	\begin{tikzpicture}
	\matrix[matrix of math nodes,row sep=3em,column sep=4em,text height=1.5ex,text depth=0.25ex]
	{|[name=A]| C_{\m}\lotimes_{R_{\m}}\rhom_R(C,X)_{\m} & |[name=B]| \left[C\lotimes_R\rhom_R(C,X)\right]_{\m}\\
	|[name=C]| C_{\m}\lotimes_{R_{\m}}\rhom_{R_{\m}}(C_{\m},X_{\m}) & |[name=D]| X_{\m}.\\};
	\draw[->,font=\scriptsize]
		(A) edge node[auto]{$\simeq$} (B)
		(A) edge node[auto]{$\simeq$} (C)
		(B) edge node[auto]{$(\xi^C_X)_{\m}$} (D)
		(C) edge node[auto]{$\xi^{C_{\m}}_{X_{\m}}$} (D);
	\end{tikzpicture}
	\end{center}
	As $X_{\m}\in \catb_{C_{\m}}(R_{\m})$ for all $\m\in \operatorname{Max}(R)$, the morphism $\xi^{C_{\m}}_{X_{\m}}$ is an isomorphism for all $\m\in \operatorname{Max}(R)$. Commutativity of the above diagram now forces $(\xi^C_X)_{\m}$ to be an isomrophism for all $\m\in \operatorname{Max}(R)$. Therefore $\xi^C_X$ is an isomorphism.

	It remains to show that $\rhom_R(C,X)\in \D_{\operatorname{b}}(R)$. As $\rhom_R(C,X)\in \D_{-}(R)$, it suffices to show that $\rhom_R(C,X)\in \D_{+}(R)$. By assumption $X_{\m}\in \catb_{C_{\m}}(R_{\m})$. Then for all $\m\in \operatorname{Max}(R)$ we have
	\begin{align*}
		\inf(\rhom_R(C,X)_{\m}) &= \inf(\rhom_{R_{\m}}(C_{\m},X_{\m}))\\
		&\geq \inf(X_{\m}) - \sup(C_{\m})\\
		&\geq \inf(X) - \sup(C)
	\end{align*}
	where the equality is by the isomorphism $\rhom_R(C,X)_{\m} \simeq \rhom_{R_{\m}}(C_{\m},X_{\m})$, the first inequality is by \cite[Proposition 4.8(c)]{christensen:scatac}, and the second inequality is by properties of localization. Thus $\inf(\rhom_R(C,X)) \geq \inf(X) - \sup(C)>-\infty$.

	For the implication (vi) $\Rightarrow$ (v), assume $X_{\m}\in\catb_{C_{\m}}(R_{\m})$ for all $\m\in \operatorname{Supp}_R(X) \cap \operatorname{Max}(R)$. Then for all $\m\in \operatorname{Max}(R) \setminus \operatorname{Supp}_R(X)$ we have $X_{\m}\simeq 0\in \catb_{C_{\m}}(R_{\m})$, as desired.
\end{proof}

The following is proven similarly to Lemma \ref{141028:1}

\begin{lem}
\label{141029:1}
	Let $X\in \D_{\operatorname{b}}(R)$. The following conditions are equivalent:
	\begin{enumerate}[\rm(i)]
	\item $X\in \catac(R)$;
	\item $U^{-1}X\in \cata_{U^{-1}C}(U^{-1}R)$ all multiplicatively closed subsets $U\subset R$;
	\item $X_{\p}\in \cata_{C_{\p}}(R_{\p})$ for all $\p\in \spec(R)$;
	\item $X_{\p}\in \cata_{C_{\p}}(R_{\p})$ for all $\p\in \operatorname{Supp}(R)$;
	\item $X_{\m}\in \cata_{C_{\m}}(R_{\m})$ for all $\m\in \operatorname{Max}(R)$; and
	\item $X_{\m}\in \cata_{C_{\m}}(R_{\m})$ for all $\m\in \operatorname{Supp}(R)\cap\operatorname{Max}(R)$.
	\end{enumerate}
\end{lem}

\begin{prop}\label{140623:10}
	Let $X\in \D_{\operatorname{b}}(R)$ and let $n\in \bbz$. Consider the following conditions:
	\begin{enumerate}[\rm (i)]
	\item $\pcpd_R(X) -\sup(C) \leq n$; \label{140623:10i}
	\item $\mathcal{P}_{U^{-1}C}\text{-}\pd_{U^{-1}R}(U^{-1}X) - \sup(U^{-1}C) \leq n$ for each multiplicatively closed subset $U\subset R$;\label{140623:10ii}
	\item $\mathcal{P}_{C_{\p}}\text{-}\pd_{R_{\p}}(X_{\p}) - \sup(C_{\p})\leq n$ for each $\p\in \spec(R)$; and \label{140623:10iii}
	\item $\mathcal{P}_{C_{\m}}\text{-}\pd_{R_{\m}}(X_{\m}) - \sup(C_{\m}) \leq n$ for each $\m\in \operatorname{Max}(R)$. \label{140623:10iv}
	\end{enumerate}
	Then \eqref{140623:10i} $\Rightarrow$ \eqref{140623:10ii} $\Rightarrow$ \eqref{140623:10iii} $\Rightarrow$ \eqref{140623:10iv}. Furthermore, if $X\in\D_{\operatorname{b}}^{\operatorname{f}}(R)$, then \eqref{140623:10iv} $\Rightarrow$ \eqref{140623:10i} and
	\begin{align*}
		\pcpd_R(X) - c
			&= \sup\left\{
			\begin{array}{l}
				\catp_{U^{-1}C}\text{-}\pd_{U^{-1}R}(U^{-1}X)\\
				- \sup(U^{-1}C)
			\end{array}
			\left|
			\begin{array}{l}
				U\subset R \text{ is}\\ 
				\text{multiplicatively}\\
				\text{closed}
			\end{array}\right.\right\}\\
			&= \sup\{\catp_{C_{\p}}\text{-}\pd_{R_{\p}}(X_{\p}) - \sup(C_{\p}) \mid \p \in \spec(R)\}\\
			&= \sup\{\catp_{C_{\m}}\text{-}\pd_{R_{\m}}(X_{\m}) - \sup(C_{\m}) \mid \m\in \mathrm{Max}(R)\}
	\end{align*}
	where $c= \sup(C)$.
\end{prop}

\begin{proof}
	Observe that \eqref{140623:10i} $\Rightarrow$ \eqref{140623:10ii} follows from Proposition \ref{141021:1}. The implications \eqref{140623:10ii} $\Rightarrow$ \eqref{140623:10iii} $\Rightarrow$ \eqref{140623:10iv} follow from properties of localization. For the rest of the proof assume that $X\in \D_{\operatorname{b}}^{\operatorname{f}}(R)$.

	For the implication (iv) $\Rightarrow$ (i) assume that $\catp_{C_{\m}}\text{-}\pd_{R_{\m}}(X_{\m}) - \sup(C_{\m}) \leq n<\infty$ for all $\m\in \operatorname{Max}(R)$. Then by Remark \ref{140623:5} we have $X_{\m}\in \catb_{C_{\m}}(R_{\m})$ for all $\m\in \operatorname{Max}(R)$. Therefore Lemma \ref{141028:1} implies that $X\in \catbc(R)$ and hence $\rhom_R(C,X)\in \D_{\operatorname{b}}(R)$. Now
	\begin{align*}
		\pcpd_R(X) - \sup(C) &= \pd_{R}(\rhom_R(C,X))\\
		&= \sup_{\m\in \operatorname{Max}(R)}(\pd_{R_{\m}}(\rhom_{R_{\m}}(C_{\m},X_{\m})))\\
		& \leq n
	\end{align*}
	where the second equality is by \cite[Proposition 5.3P]{avramov:hdouc}.

	For the equalities, assume first that $\pcpd_R(X) -\sup(C) = n<\infty$. Then each displayed supremum in the statement is at most $n$. If any of the supremums are strictly less than $n$, then the above equivalence will force $\pcpd_R(X) - \sup(C)<n$, contradicting our assumption. A similar argument establishes the desired equalities if we assume any of the supremums equal $n$.

	Finally if any of the displayed values in the statement are infinite, then the above equivalences forces the other values to be infinite as well.
\end{proof}

	To prove the implication (iv) $\Rightarrow$ (i) in Proposition \ref{140623:10}, the condition $X\in \D_{\operatorname{b}}^{\operatorname{f}}(R)$ is required. However the flat and injective versions only require $X\in \D_{\operatorname{b}}(R)$; see \cite[Propositions 5.3F,5.3I]{avramov:hdouc}. Thus the next two results are proven similarly to Proposition \ref{140623:10}.

\begin{prop}\label{140623:11}
	Let $X\in \D_{\operatorname{b}}(R)$ and let $n\in \bbz$. The following conditions are equivalent:
	\begin{enumerate}[\rm (i)]
	\item $\fcpd_R(X) -\sup(C) \leq n$; \label{140623:11i}
	\item $\mathcal{F}_{U^{-1}C}\text{-}\pd_{U^{-1}R}(U^{-1}X) - \sup(U^{-1}C) \leq n$ for each multiplicatively closed subset $U\subset R$;\label{140623:11ii}
	\item $\mathcal{F}_{C_{\p}}\text{-}\pd_{R_{\p}}(X_{\p}) - \sup(C_{\p})\leq n$ for each prime ideal $\p\subset R$; and \label{140623:11iii}
	\item $\mathcal{F}_{C_{\m}}\text{-}\pd_{R_{\m}}(X_{\m}) - \sup(C_{\m}) \leq n$ for each maximal ideal $\m\subset R$. \label{140623:11iv}
	\end{enumerate}
	Furthermore 
	\begin{align*}
		\fcpd_R(X) - c
			&= \sup\left\{
			\begin{array}{l}
				\catf_{U^{-1}C}\text{-}\pd_{U^{-1}R}(U^{-1}X)\\
				- \sup(U^{-1}C)
			\end{array}
			\left|
			\begin{array}{l}
				U\subset R \text{ is}\\ 
				\text{multiplicatively}\\
				\text{closed}
			\end{array}\right.\right\}\\
			&= \sup\{\catf_{C_{\p}}\text{-}\pd_{R_{\p}}(X_{\p}) - \sup(C_{\p}) \mid \p \in \spec(R)\}\\
			&= \sup\{\catf_{C_{\m}}\text{-}\pd_{R_{\m}}(X_{\m}) - \sup(C_{\m}) \mid \m\in \mathrm{Max}(R)\}
	\end{align*}
	where $c= \sup(C)$.
\end{prop}

\begin{prop}\label{140623:12}
	Let $X\in \D_{\operatorname{b}}(R)$ and let $n\in \bbz$. The following conditions are equqivalent:
	\begin{enumerate}[\rm (i)]
	\item $\icid_R(X) -\sup(C) \leq n$; \label{140623:12i}
	\item $\mathcal{I}_{U^{-1}C}\text{-}\id_{U^{-1}R}(U^{-1}X) - \sup(U^{-1}C) \leq n$ for each multiplicatively closed subset $U\subset R$;\label{140623:12ii}
	\item $\mathcal{I}_{C_{\p}}\text{-}\id_{R_{\p}}(X_{\p}) - \sup(C_{\p})\leq n$ for each prime ideal $\p\subset R$; and \label{140623:12iii}
	\item $\mathcal{I}_{C_{\m}}\text{-}\id_{R_{\m}}(X_{\m}) - \sup(C_{\m}) \leq n$ for each maximal ideal $\m\subset R$. \label{140623:12iv}
	\end{enumerate}
	Furthermore 
	\begin{align*}
		\icid_R(X) - c
			&= \sup\left\{
			\begin{array}{l}
				\id_{U^{-1}R}(U^{-1}C\lotimes_{U^{-1}R}U^{-1}X)\\
				- \sup(U^{-1}C)
			\end{array}
			\left|
			\begin{array}{l}
				U\subset R \text{ is}\\ 
				\text{multiplicatively}\\
				\text{closed}
			\end{array}\right.\right\}\\
			&= \sup\{\id_{R_{\p}}\text{-}\id_{R_{\p}}(C_{\p}\lotimes_{R_{\p}} X_{\p}) - \sup(C_{\m}) \mid \p \in \spec(R)\}\\
			&= \sup\{\id_{R_{\m}}\text{-}\id_{R_{\m}}(C_{\m}\lotimes_{R_{\m}} X_{\m}) - \sup(C_{\m}) \mid \m\in \mathrm{Max}(R)\}
	\end{align*}
	where $c= \sup(C)$.
\end{prop}

\begin{rem}
\label{141029:3}
	When $C$ is a semidualizing $R$-module, e.g., $C=R$, we recover the known local-global conditions for $\pcpd$, $\fcpd$, $\icid$, $\pd$, $\fd$, and $\id$.
\end{rem}

\section{Stability Results}\label{141111:5}

In this section we investigate the behaviour of $\pcpd$, $\fcpd$, and $\icid$ after applying the functors $\lotimes$ and $\rhom$.

\begin{prop}
\label{140822:1}
	Let $X,Y\in\D_{\mathrm{b}}(R)$. The following inequalities hold:
	\begin{enumerate}[\rm(a)]
	\item $\pcpd_R(X\lotimes_R Y)\leq \pcpd_R(X) + \pd_R(Y)$; \label{140822:1a}
	\item $\icid_R(\rhom_R(X,Y))\leq \fcpd_R(X) + \id_R(Y)$; and \label{140822:1b}
	\item $\fcpd_R(X\lotimes_R Y) \leq \fcpd_R(X)+\fd_R(Y)$. \label{140822:1c}
	\end{enumerate}
\end{prop}

\begin{proof}
	(a) Without loss of generality we assume that $\pcpd_R(X)<\infty$ and $\pd_R(Y)<\infty$. It now follows that $\pcpd_R(X) = \sup(C)+ \pd_R(\rhom_R(C,X))$. By \cite[Theorem 4.1 (P)]{avramov:hdouc} we have that
	\[
		\pd_R(\rhom_R(C,X)\lotimes_R Y)\leq \pd_R(\rhom_R(C,X)) + \pd_R(Y).
	\]
	Since $\pd_R(Y)<\infty$ (hence $\fd_R(Y)<\infty$) we get tensor-evaluation (\ref{141009:1}) is an isomorphism in $\D(R)$. That is $\rhom_R(C,X\lotimes_R Y)\simeq \rhom_R(C,X)\lotimes_R Y$. Hence we have
	\[
		\pd_R(\rhom_R(C,X\lotimes_R Y))\leq \pd_R(\rhom_R(C,X))+\pd_R(Y).
	\]
	By adding a $\sup(C)$ to each side we see that $\pcpd_R(X\lotimes_RY)\leq \pcpd_R(X) + \pd_R(Y)$.

	(b) and (c) are proven similarly to (a).
\end{proof}

\begin{cor}
\label{140822:2}
	Let $X\in\D_{b}(R)$. The following inequalities hold:
	\begin{enumerate}[\rm(a)]
	\item $\pcpd_R(X\lotimes_R \rhom_R(C,Y))\leq \pcpd_R(X) + \pcpd_R(Y)-\sup(C)$; \label{140822:2a}
	\item $\icid_R(\rhom_R(X,C\lotimes_R Y))\leq \fcpd_R(X) + \icid_R(Y)-\sup(C)$; and \label{140822:2b}
	\item $\fcpd_R(X\lotimes_R \rhom_R(C,Y))\leq \fcpd(RX) + \fcpd_R(Y) - \sup(C)$. \label{140822:2c}
	\end{enumerate}
\end{cor}

\begin{proof}
	\eqref{140822:2a} By Proposition \ref{140822:1}\eqref{140822:1a} we have that $\pcpd_R(X\lotimes_R\rhom_R(C,Y))\leq \pcpd_R(X) + \pd_R(\rhom_R(C,Y))$. Add and subtract $\sup(C)$ to the right hand side to obtain the result. 
	
	\eqref{140822:2b} and \eqref{140822:2c} are proven similarly.
\end{proof}

The next result is a version of Fact \ref{141121:1} involving a semidualizing complex.

\begin{prop}
\label{140827:2}
	Let $X,Y\in \D_{\mathrm{b}}(R)$.
	\begin{enumerate}[\rm (a)]
	\item If $\id_R(Y)<\infty$, then $\fcpd_R(\rhom_R(X,Y))\leq \icid_R(X) + \sup(Y)$. \label{140827:2a}
	\item If $\fd_R(Y)<\infty$, then $\icid_R(X\lotimes_R Y)\leq \icid_R(X) - \inf(Y)$. \label{140827:2b}
	\end{enumerate}
\end{prop}

\begin{proof}
	\eqref{140827:2a} Assume that $\id_R(Y)<\infty$. By applying Defintion \ref{140226:1} we get that $\fcpd_R(\rhom_R(X,Y)) = \sup(C) + \fd_R(\rhom_R(C,\rhom_R(X,Y)))$. Observe by Hom-Tensor adjointness there is an isomorphism
	\[
		\rhom_R(C,\rhom_R(X,Y))\simeq \rhom_R(C\lotimes_R X,Y).
	\]
	Therefore $\fd_R(\rhom_R(C,\rhom_R(X,Y)) = \fd_R(\rhom_R(C\lotimes_R X,Y))$. Hence by Fact \ref{141121:1}\eqref{141121:1a} we have that
	\[
		\fd_R(\rhom_R(C\lotimes_R X,Y)\leq \id_R(C\lotimes_R X) + \sup(Y).
	\]
	By adding $\sup(C)$ to each side of the above inequality we obtain the desired result.

	\eqref{140827:2b} is proven similarly.
\end{proof}

\begin{prop}
\label{140313:2}
	Let $X\in \mathcal{D}_{\mathrm{b}}(R)$. The following conditions are equivalent:
	\begin{enumerate}[\rm(i)]
	\item $\fcpd_R(X)<\infty$;
	\item $\icid_R(\rhom_R(X,Y))<\infty$ for all $Y\in \D_{\mathrm{b}}(R)$ such that $\id_R(Y)<\infty$; and
	\item $\icid_R(\rhom_R(X,E))<\infty$ for some faithfully injective $R$-module $E$.
	\end{enumerate}
\end{prop}

\begin{proof}
	(i)$\Rightarrow$(ii) This follows from Proposition \ref{140822:1}\eqref{140822:1b}.

	(ii)$\Rightarrow$(iii) Since $E$ is a faithfully injective module it has $\id_R(E)=0<\infty$. Therefore (ii) implies that $\icid_R(\rhom_R(X,E))<\infty$.

	(iii)$\Rightarrow$(i) Assume that there exists a faithfully injective $R$-module $E$ such that $\icid_R(\rhom_R(X,E))<\infty$. Then by Definition \ref{140226:1}\eqref{140226:1c} $\icid_R(\rhom_R(X,E)) = \sup(C) + \id_R(C\lotimes_R \rhom_R(X,E))$. By Hom-evaluation (\ref{141009:1}) there is an isomorphism 
	\[
		\rhom_R(\rhom_R(C,X),E) \simeq C\lotimes_R \rhom_R(X,E).
	\]
	It follows that $\id_R(C\lotimes_R \rhom_R(X,E)) = \id_{R}(\rhom_R(\rhom_R(C,X),E))<\infty$. Therefore by Lemma \ref{140821:1} $\fd_R(\rhom_R(C,X))<\infty$. It now follows that $\fcpd_R(X)<\infty$.
\end{proof}

The following three propositions are proven similarly to Proposition \ref{140313:2}.

\begin{prop}
\label{141001:2}
	Let $X\in \D_{\operatorname{b}}(R)$. The following conditions are equivalent:
	\begin{enumerate}[\rm(i)]
	\item $\fcpd_R(X)<\infty$;
	\item $\fcpd_R(X\lotimes_R Y)<\infty$ for all $Y\in \D_{\operatorname{b}}(R)$ such that $\fd_R(Y)<\infty$;
	\item $\fcpd_R(X\lotimes_R F)<\infty$ for some faithfully flat $R$-module $F$.
	\end{enumerate}
\end{prop}

\begin{prop}
\label{141007:2}
	Let $X\in \D_{\operatorname{b}}(R)$. The following conditions are equivalent:
	\begin{enumerate}[\rm(i)]
	\item $\icid_R(X)<\infty$;
	\item $\fcpd_R(\rhom_R(X,Y))<\infty$ for all $Y\in \D_{\operatorname{b}}(R)$ such that $\id_R(Y)<\infty$;
	\item $\fcpd_R(\rhom_R(X,E))<\infty$ for some faithfully injective $R$-module $E$.
	\end{enumerate}
\end{prop}

\begin{prop}
\label{141008:1}
	Let $X\in \D_{\operatorname{b}}(R)$. The following conditions are equivalent:
	\begin{enumerate}[\rm(i)]
	\item $\icid_R(X)<\infty$;
	\item $\icid_R(X\lotimes_R Y)<\infty$ for all $Y\in \D_{\operatorname{b}}(R)$ such that $\fd_R(Y)<\infty$;
	\item $\icid_R(X\lotimes_R F)<\infty$ for some faithfully flat $R$-module $F$.
	\end{enumerate}
\end{prop}

\begin{cor}
\label{140311:2}
	Let $X\in \mathcal{D}_{\rm{b}}(R)$ and. If there exists a dualizing complex $D$ and $\fcpd_R(X)<\infty$, then  $\icid_R(X^{\dagger})<\infty$ where $X^{\dagger} = \rhom_R(X,D)$.
\end{cor}

\begin{proof}
	Since $D$ is a dualizing complex, it has finite injective dimension. Therefore the result follows from Proposition \ref{140313:2}.
\end{proof}

The last result of this paper establishes Theorem \ref{141117:1} from the introduction.

\begin{thm}\label{140503:1}
	Assume $R$ has a dualizing complex $D$ and let $X\in \mathcal{D}_{\rm{b}}(R)$. Then $\icid_R(X)<\infty$ if and only if $\mathcal{F}_{C^{\dagger}}\text{-}\pd_R(X)<\infty$ where $C^{\dagger} = \rhom_R(C,D)$.
\end{thm}

\begin{proof}
	For the forward implication assume that $\icid_R(X)<\infty$. Then set $J = C\lotimes_R X$. Since $\icid_R(X)<\infty$ we have that $J$ has finite injective dimension. By Remark \ref{140623:5} we have $X\in\catac(R)$. This explains the first isomorphism in the following display:
	\[
		X\simeq \rhom_R(C,J) \simeq \rhom_R(\rhom_R(C^{\dagger},D),J) \simeq C^{\dagger}\lotimes_R \rhom_R(D,J).
	\]
	The second isomorphism is from the isomorphism $C\simeq C^{\dagger\dagger}$, and the third is by Hom-evaluation (\ref{141009:1}). Observe that since $\id_R(D)<\infty$ and $\id_R(J)<\infty$ we have that $\fd_R(\rhom_R(D,J))<\infty$ by Fact \ref{141121:1}\eqref{141121:1a}. Thus, it follows that $\mathcal{F}_{C^{\dagger}}\text{-}\fd_R(X)<\infty$ by the displayed isomorphisms.

	For the reverse implication assume that $\mathcal{F}_{C^{\dagger}}\text{-}\fd_R(X)<\infty$. Then we can write $X\simeq C^{\dagger}\lotimes_R F$ where $F= \rhom_R(C^{\dagger},X)$ and $\fd_R(F)<\infty$. We then have the following isomorphisms:
	\[
		X\simeq C^{\dagger}\lotimes_R F = \rhom_R(C,D)\lotimes_R F \simeq \rhom_R(C,D\lotimes_R F)
	\]
	where the second isomorphism is by tensor-evaluation (\ref{141009:1}). Since $\id_R(D)<\infty$ and $\fd_R(F)<\infty$ we have that $\id_R(D\lotimes_R F)<\infty$ by Fact \ref{141121:1}\eqref{141121:1b}. Hence $\icid_R(X)<\infty$ by Theorem \ref{141111:6}\eqref{141111:6c} as desired.
\end{proof}


\end{document}